\documentclass{article}
\usepackage{amsmath}
\usepackage{amsfonts}
\usepackage{amsthm}
\usepackage{amssymb}
\usepackage{color}

\newtheorem{theorem}{Theorem}[section]
\newtheorem{lemma}[theorem]{Lemma}

\newtheorem{observation}[theorem]{Observation}
\newtheorem{problem}[theorem]{Problem}

\theoremstyle{definition}
\newtheorem{definition}[theorem]{Definition}
\theoremstyle{remark}
\newtheorem{remark}[theorem]{Remark}

\newcommand{\vol}{{\rm vol}\hskip0.02cm}

\def\f2{\mathbb{F}_2}

\def\lip{\hskip0.02cm{\rm Lip}\hskip0.01cm}

\newcommand{\ep}{\varepsilon}

\begin{document}
\title{\LARGE{Different forms of metric
characterizations of classes of Banach spaces}}

\author{M.\,I.~Ostrovskii}

\date{\today}
\maketitle

\noindent{\bf Abstract.} For each sequence $\{X_m\}_{m=1}^\infty$
of finite-dimensional Banach spaces there exists a sequence
$\{H_n\}_{n=1}^\infty$ of finite connected unweighted graphs with
maximum degree $3$ such that the following conditions on a Banach
space $Y$ are equivalent:
\begin{itemize}
\item $Y$ admits uniformly isomorphic embeddings of
$\{X_m\}_{m=1}^\infty$. \item $Y$ admits uniformly bilipschitz
embeddings of $\{H_n\}_{n=1}^\infty$.
\end{itemize}
\medskip

\noindent{\bf 2010 Mathematics Subject Classification:} Primary:
46B07; Secondary: 05C12, 46B85, 54E35

\begin{large}


\section{Introduction}

In connection with  problems of embeddability of metric spaces
into Banach spaces it would be interesting to find metric
characterizations of well-known classes of Banach spaces. By a
{\it metric characterization} we mean a set of formulas with some
variables, quantifiers and inequalities, where the inequalities
are between algebraic expressions containing distances between
those variables which are elements of spaces. We say that such set
of formulas {\it characterizes} a class $\mathcal{P}$ of Banach
spaces if $X\in\mathcal{P}$ if and only if all of the formulas of
the set hold for $X$. We consider a more narrow class of metric
characterizations, namely characterizations based on the notion of
test-spaces.

\begin{definition}\label{D:TestSp} Let $\mathcal{P}$ be a class of Banach spaces
and let $T=\{T_\alpha\}_{\alpha\in A}$ be a set of metric spaces.
We say that $T$ is a set of {\it test-spaces} for $\mathcal{P}$ if
the following two conditions are equivalent
\begin{enumerate}
\item $X\notin\mathcal{P}$

\item\label{I:DefTesSp2} The spaces $\{T_\alpha\}_{\alpha\in A}$
admit bilipschitz embeddings into $X$ with uniformly boun\-ded
distortions.
\end{enumerate}
\end{definition}

\begin{remark} Reading of the rest of this introduction requires
more background than reading Sections
\ref{S:Graphic}--\ref{S:StricConv} of the paper.
\end{remark}

\begin{remark} We write $X\notin\mathcal{P}$ rather than
$X\in\mathcal{P}$ for terminological reasons: we would like to use
terms ``test-spaces for reflexivity, superreflexivity, etc.''
rather than ``test-spaces for {\bf non}reflexivity, {\bf
non}superreflexivity, etc.''
\end{remark}

\begin{remark} For each collection $\{T_\alpha\}_{\alpha\in A}$ of metric spaces
the condition \eqref{I:DefTesSp2} of Definition \ref{D:TestSp}
determines the corresponding class of Banach spaces. It seems that
at least in some cases it would be interesting to start with a
class of metric spaces and to try to understand what is the
corresponding class of Banach spaces. Papers \cite[Corollary
1.7]{CK09} and \cite[Theorems 3.2 and 3.6]{Ost10+} contain results
of this type. We are not going to pursue this direction here.
\end{remark}

Test-spaces are known for several important classes of Banach
spaces: superreflexive \cite{Bou86,Mat99,Bau07,JS09}; spaces
having some type $t>p$, where $p\in[1,2)$
\cite{BMW86,Pis86,Bau07}; spaces with non-trivial cotype
\cite{MN08,Bau09+}.
\medskip

There are some test-space characterizations which are usually
stated differently and considered as Ba\-nach-space-theoretical
characterizations rather than metric characterizations. We mean
the following characterizations. We refer to \cite{DJT95} and
\cite{LT79} for the theory of type and cotype of Banach spaces.
\medskip

Maurey and Pisier \cite{MP76} introduced for each
infinite-dimensional Banach space $X$ parameters
$p(X)=\sup\{p:~X\hbox{ has type }p\}$ and $q(X)=\inf\{q:~X\hbox{
has cotype }q\}$ and proved that the spaces
$\{\ell_{p(X)}^n\}_{n=1}^\infty$ and
$\{\ell_{q(X)}^n\}_{n=1}^\infty$ are finitely representable in $X$
in the sense that $\forall\ep>0$ $\forall n\in\mathbb{N}$ there is
a subspace $X_{n,\ep}$ in $X$ satisfying
$d(X_{n,\ep},\ell_{p(X)}^n)\le1+\ep$, where $d$ is the
Banach-Mazur distance. (This result is based on the work of
Krivine \cite{Kri76}, proofs of these results can be found in
\cite{MS86}.)\medskip

On the other hand, Bretagnolle, Dacunha-Castelle, and Krivine
\cite{BDK66} proved that for $1\le p<q\le 2$ the space $L_q$
admits an isometric embedding into $L_p$.
\medskip

Also, as is easy to check, the space $L_p$ $(1\le p\le 2)$ has
type $p$ but does not have any type $t>p$ (see
\cite[p.~216]{DJT95}).
\medskip

Combining these results we get the following characterizations.

\begin{itemize}

\item[{\bf (i)}] Let $p\in[1,2)$. Let $\mathcal{A}_p$ be the class
of Banach spaces for which $p(X)>p$. Then $X\notin \mathcal{A}_p$
if and only if $\{\ell_p^n\}_{n=1}^\infty$ admit uniformly
isomorphic embeddings into $X$.

\item[{\bf (ii)}] In a similar way spaces with $q(X)=\infty$ are
characterized as spaces admitting uniformly isomorphic embeddings
of $\{\ell_\infty^n\}_{n=1}^\infty$.

\end{itemize}

The characterizations just stated are not in terms of bilipschitz
embeddings, but applying the theory of differentiability, see
\cite{HM82} and \cite[Theorem 7.9 and Corollary 7.10]{BL00} and
local reflexivity (see \cite{LR69} and \cite{JRZ71}) it can be
shown that the word ``isomorphic'' can be replaced by the word
``bilipschitz'' in {\bf (i)--(ii)}. (We say that a collection of
embeddings is {\it uniformly bilipschitz} if the embeddings have
uniformly bounded distortions.) We do not present this argument in
detail here because it is almost identical to the argument which
we present at the end of our proof of Theorem \ref{T:FinDimFinGr}.
\medskip

In this paper we are interested in the following problem.

\begin{problem}\label{P:GraphTestSp} Let
$\mathcal{P}$ be a class of Banach spaces which can be
characterized in terms of countably many test-spaces which are
finite-dimensional normed spaces.
\smallskip

\noindent{\bf (a)} Is it possible to describe $\mathcal{P}$ in
terms of countably many test-spaces which are finite unweighted
graphs with their graph distances?
\smallskip

\noindent{\bf (b)} Is it possible to require, in addition, that
the graphs have uniformly bounded degrees of vertices?
\end{problem}

Special cases of this problem were posed by W.\,B.~Johnson during
the seminar ``Nonlinear geometry of Banach spaces'' (Workshop in
Analysis and Probability at Texas A~\&~M University,
2009).\medskip

Our main purpose it to give an affirmative answer to Problem
\ref{P:GraphTestSp}.

\section{Graphic test-spaces for classes having finite-dimensional
test-spaces}\label{S:Graphic}

Our first goal is to prove the following result:

\begin{theorem}\label{T:FinDimFinGr} If a class $\mathcal{P}$ can be characterized
using test-spaces $\{X_m\}_{m=1}^\infty$ which are finite
dimensional Banach spaces, then $\mathcal{P}$ can be characterized
using test-spaces which are finite unweighted graphs with their
graph distances.
\end{theorem}

By a $\delta${\it-net} in a metric space $Z$ we mean a collection
$U$ of elements of $Z$ satisfying the conditions:
\begin{itemize}
\item $\forall z\in Z~ \exists u\in U~ d_Z(u,z)\le\delta$. \item
$\forall u,v\in U~ d_Z(u,v)\ge\delta$.
\end{itemize}

\begin{lemma}\label{L:GraphInBall} For each finite-dimensional Banach space $X$ and each
pair $\delta,r$ satisfying $0<\delta<r<\infty$ there is a finite
unweighted graph $G=(V(G),E(G))$ and a map $f:V(G)\to rB(X)$
($rB(X)$ is a multiple of the unit ball) such that $f$ is a
bilipschitz embedding with distortion $\le 3$ and $f(V(G))$ is a
$\delta$-net in $rB(X)$ with distances $>\delta$ between images of
different vertices.
\end{lemma}

\begin{remark} It is interesting to compare this lemma with the observation that an unweighted graph $G$ which admits
an isometric embedding into a strictly convex Banach space should
be either a path or a complete graph. Most probably this
observation is known. I enclose a proof of it in Section
\ref{S:StricConv} because I have not found it in the literature.
\end{remark}

\begin{proof}[Proof of Lemma \ref{L:GraphInBall}]
Since $X$ is finite-dimensional, there is a finite $\delta$-net
$V=\{v_i\}_{i=1}^n$ in $rB(X)$ with $||v_i-v_j||>\delta$. We
introduce a graph structure on $V$ using the following rule:
vertices $v_i$ and $v_j$ are adjacent if and only if
$||v_i-v_j||\le 3\delta$. Denote the obtained graph by
$G=G(X,\delta,r)$.
\medskip

Let $f:V\to X$ be the natural  embedding (that is, embedding which
maps each vertex onto itself). It is clear that $\lip(f)\le
3\delta$ and that $f(V)$ is a $\delta$-net in $rB(X)$. The only
condition which is to be verified is $\lip(f^{-1})\le\delta^{-1}$.
It suffices to show that for each pair $u,v\in V$ satisfying
$||u-v||=d\delta$, there is a $uv$-path having at most $\lfloor
d\rfloor$ edges.
\medskip

If $1< d\le 3$, the statement is obvious since $u$ and $v$ are
adjacent  (observe that $d$ cannot be $\le 1$). If $d>3$ we use
the following lemma.

\begin{lemma}\label{L:ShortInNet} Let $u,v\in V$ be such that
$||u-v||=d\delta>3\delta$. Then there is a vertex $u_1\in V$
satisfying $||u-u_1||\le 3\delta$ and $||u_1-v||\le (d-1)\delta$.
\end{lemma}

\begin{proof} Let $w$ be the point satisfying $||u-w||=2\delta$ and belonging to the line segment joining $u$
and $v$. Then there is $u_1\in V$ satisfying $||u_1-w||\le\delta$.
By the triangle inequality we have $||u-u_1||\le 3\delta$ and
$||v-u_1||\le||v-w||+||u_1-w||\le
d\delta-2\delta+\delta=(d-1)\delta$.
\end{proof}

We complete the proof of Lemma \ref{L:GraphInBall} using
induction. We know that the claim holds when $||u-v||<3\delta$.
\medskip

\noindent{\bf Induction Hypothesis:} The claim holds when
$||u-v||<n\delta$.
\medskip

Now assume that $n\delta\le||u-v||<(n+1)\delta$. We apply Lemma
\ref{L:ShortInNet} and get $u_1$ satisfying $||u-u_1||\le 3\delta$
and $||u_1-v||\le ||u-v||-\delta<n\delta$. By the Induction
Hypothesis there is a $u_1v$-path of length $\le \lfloor
||u_1-v||/\delta\rfloor$. Also $u$ and $u_1$ are adjacent in $G$.
Adding this edge to the $u_1v$-path we get a $uv$-path of length
$\le \lfloor ||u_1-v||/\delta\rfloor+1\le\lfloor
||u-v||/\delta\rfloor$.\end{proof}

\begin{proof}[Proof of Theorem \ref{T:FinDimFinGr}] Our purpose is
to show that the countable collection
\[\{G(X_m,1/n,n)\}_{m,n=1}^\infty\] of graphs
 is the desired collection of
test-spaces, where $\{X_m\}_{m=1}^\infty$ are
finite-di\-men\-sio\-nal Banach test-spaces for
$\mathcal{P}$.\medskip

It is clear that the spaces $\{G(X_m,\frac1n,n)\}_{m,n=1}^\infty$
admit uniformly bilipschitz embeddings into any Banach space
admitting uniformly bilipschitz embeddings of
$\{X_m\}_{m=1}^\infty$.
\medskip

It remains to show that if a Banach space $Y$ admits uniformly
bilipschitz embeddings of $\{G(X_m,\frac1n,n)\}_{m,n=1}^\infty$,
then there are uniformly isomorphic embeddings of the Banach
spaces $\{X_m\}$ into $Y$. This can be proved in the following way
(the author learned this argument from G.~Schechtman, see
\cite[Proposition 4.2]{Bau09+}).\medskip

Fix $m\in\mathbb{N}$. Let $f_n:G(X_m,\frac1n,n)\to Y$ be uniformly
bilipschitz embeddings. We may assume that there is $0<C<\infty$
such that $d_n(u,v)\le ||f_n(u)-f_n(v)||\le Cd_n(u,v)$ for all
vertices $u,v$ of $G(X_m,\frac1n,n)$, where $d_n$ is the graph
distance of $G(X_m,\frac1n,n)$. We may and shall assume that the
zero element of $X_m$ is a vertex of $G(X_m,\frac1n,n)$, and that
$f_n(0)=0$, where the first $0$ is the zero element of $X_m$ and
the second $0$ is the zero element in $Y$.
\medskip

We use the embeddings $f_n$ to find a $C$-bilipschitz embedding of
$X_m$ into an ultrapower of $X$. For each $y$ in $X_m$ we
introduce the following sequence $y_n\in V(G(X_m,\frac1n,n))$,
$n\in\mathbb{N}$:
\begin{equation}\label{E:y_n}y_n=\begin{cases} 0 & \hbox{ if } ||y||>n\\
\hbox{best approximation of $y$ by elements of
$V(G(X_m,\frac1n,n))$}& \hbox{ if } ||y||\le n.
\end{cases}\end{equation}
In this definition, we pick one of the best approximations if
there are several of them.
\medskip

Let $\mathcal{U}$ be a free ultrafilter on $\mathbb{N}$. It is
easy to check that the mapping $F:X_m\to Y$ given by
\[F(y)=\left\{\frac 1n\,f_n(y_n)\right\}_{n=1}^\infty\]
is a $3C$-bilipschitz embedding of $X_m$ into $(Y)_\mathcal{U}$
and thus into the second dual $((Y)_\mathcal{U})^{**}$. By
\cite[Theorem 7.9 and Corollary 7.10]{BL00} (these results go back
to \cite{HM82}), this implies an existence of $C$-isomorphic
embedding of $X_m$ into $((Y)_\mathcal{U})^{**}$. Using the local
reflexivity (\cite{LR69} and \cite{JRZ71}) and standard properties
of ultraproducts (see \cite{DK72} or \cite{DJT95}) we get that
$\{X_m\}_{m=1}^\infty$ are uniformly isomorphic to subspaces of
$Y$.
\end{proof}

\section{Test-spaces with uniformly bounded degrees}

It is easy to see that for a fixed finite-dimensional space $X$
the graphs $\{G(X,\frac1n,n)\}$ have uniformly bounded degrees,
but when we consider a family of graphs corresponding to spaces
$\{X_m\}$ with growing dimensions, we get graphs with no uniform
bound on degrees. Therefore Problem \ref{P:GraphTestSp}{\bf (b)}
is nontrivial in this case. However, the answer to it is positive:

\begin{theorem}\label{T:BoundDegr} Let $\{X_m\}_{m=1}^\infty$ be a
sequence of finite-dimensional normed spaces satisfying
$\sup_{m}\dim X_m=\infty$. Then there exists a sequence
$\{H_n\}_{n=1}^\infty$ of finite unweighted graphs with maximum
degree $3$ such that a Banach space $Y$ admits uniformly
bilipschitz embeddings of $\{H_n\}_{n=1}^\infty$ if and only if
$Y$ admits uniformly bilipschitz (or uniformly isomorphic)
embeddings of $\{X_m\}_{m=1}^\infty$.
\end{theorem}

\begin{remark} It is easy to see that the result remains true if
$\sup_m\dim X_m<\infty$, but this case seems to be of little
interest.
\end{remark}

The general scheme of the proof of Theorem \ref{T:BoundDegr} is
the same as in \cite[Theorem 2.1]{Ost10+}. The main step in the
proof of Theorem \ref{T:BoundDegr} is the following lemma (its
analogues for the cases considered in \cite{Ost10+} were much
easier).

\begin{lemma}\label{L:EmbedEdges} Let $X$ be a finite-dimensional normed space with $\dim X\ge 3$ and let
$G=G(X,\delta,r)$, $0<\delta<r$, be a graph defined in the proof
of Lemma \ref{L:GraphInBall}. Let $M\in\mathbb{N}$ and let $MG$ be
the graph obtained from $G$ if each edge is replaced by a path of
length $M$. Then the graphs $\{MG\}_{M=1}^\infty$ admit uniformly
bilipschitz embeddings into $X$. Furthermore, there is an upper
bound on distortion which is an absolute constant.
\end{lemma}

\begin{proof}[First we prove Theorem \ref{T:BoundDegr} using Lemma
\ref{L:EmbedEdges}] Since the sequence $\{\dim X_m\}_{m=1}^\infty$
is unbounded, a Banach space which admits uniformly isomorphic
embeddings of $\{X_m\}_{m=1}^\infty$, admits uniformly isomorphic
embeddings of the sequence
$\{X_m\oplus_1\mathbb{R}\}_{m=1}^\infty$. It is also clear that if
we drop from the sequence $\{X_m\}$ all spaces with $\dim X_m<3$,
this would not change the class of Banach spaces admitting
uniformly isomorphic embeddings of $\{X_m\}$. Therefore we may and
shall assume that Lemma \ref{L:EmbedEdges} is applicable to each
member of the sequence $\{X_m\}$.
\medskip

Our proof of Theorem \ref{T:FinDimFinGr} implies that it suffices
to show that for each graph $G=G(X,\delta,r)$ there exist a graph
$H$ and bilipschitz embeddings $\psi:G\to H$ and $\varphi:H\to
X\oplus_1\mathbb{R}$ such that their distortions are bounded from
above by absolute constants and the maximum degree of $H$ is $3$.
It is easy to see that it is enough to consider the case
$\delta=1$.
\medskip

Our construction can be described in the following way: First we
expand $G$ replacing each edge by a path of length $M$. We use the
term {\it long paths} for these paths, the number $M$ here is
chosen to be much larger than the number of edges of $G$ (actually
we can replace the number of edges of $G$ by a smaller number in
this argument, but we do not see reasons to work on this
modification now). Then, for each vertex $v$ of $G$, we introduce
a path $p_v$ in the graph $H$ (which we are constructing now)
whose length is equal to the number of edges of $G$, we call each
such path a {\it short path}. At the moment these paths do not
interact. We continue our construction of $H$ in the following
way. We label vertices of short paths in a monotone way by long
paths. ``In a monotone way'' means that the first vertex of {\bf
each} short path corresponds to the long path $p_1$, the second
vertex of {\bf each} short path corresponds to the long path $p_2$
etc. We complete our construction of $H$ introducing, for a long
path $p$ in $MG$ corresponding to an edge $uv$ in $G$, a path of
the same length in $H$ (we also call it {\it long}) which joins
those vertices of the short paths $p_u$ and $p_v$ which have label
$p$. There is no further interaction between short and long paths
in $H$. It is obvious that the maximum degree of $H$ is $3$.
\medskip

It remains to define embeddings $\psi$ and $\varphi$ and to
estimate their distortions.\medskip

To define $\psi$ we pick a long path $p$ in $MG$ (in an arbitrary
way) and map each vertex $u$ of $G$ onto the vertex in $H$ having
label $p$ in the short path $p_u$ corresponding to $u$. We have
$\lip(\psi)\le 2e(G)+M$, where $e(G)$ is the number of edges of
$G$. In fact, to estimate the Lipschitz constant it suffices to
find an estimate from above for the distances in $H$ between
$\psi(u)$ and $\psi(v)$ where $u$ and $v$ are adjacent vertices of
$G$. To see that $2e(G)+M$ provides the desired estimate we
consider the following three-stage walk from $\psi(u)$ from
$\psi(v)$:
\begin{itemize}
\item We walk from $\psi(u)$ along the short path $p_u$ to the
vertex labelled by the long path corresponding to the edge $uv$ in
$G$. \item Then we walk along the corresponding long path to its
end in $p_v$. \item We conclude the walk with the piece of the
short path $p_v$ which we need to traverse in order to reach
$\psi(v)$.
\end{itemize}
\smallskip

We claim that $\lip(\psi^{-1})\le M^{-1}$ (this gives an absolute
upper bound for the distortion of $\psi$ provided the quantity
$e(G)$ does not exceed $M$, the assumption ``$M$ is much larger
than $e(G)$'' made above is needed only if we would like to make
the distortion close to $1$). In fact, let $\psi(u)$ and $\psi(v)$
be two vertices of $\psi(V(G))$. We need to estimate $d_G(u,v)$
from below in terms of $d_H(\psi(u),\psi(v))$. Let
$P=\psi(u),w_1,\dots,w_n=\psi(v)$ be one of the shortest
$\psi(u)\psi(v)$-paths in $H$. Let $u,u_1,\dots,u_k=v$ be those
vertices of $G$ for which the path $P$ visits the corresponding
short paths $p_u,p_{u_1},\dots,p_{u_k}=p_v$. We list
$u_1,\dots,u_k$ in the order of visits. It is clear that in such a
case $u,u_1,\dots,u_k=v$ is a $uv$-walk in $G$. Therefore
$d_G(u,v)\le k$. On the other hand, in $H$, to move from one short
path to another, one has to traverse at least $M$ edges, therefore
$d_H(\psi(u),\psi(v))\ge kM$. This implies $\lip(\psi^{-1})\le
M^{-1}$.\medskip

Now we describe $\varphi:H\to X\oplus_1\mathbb{R}$. First we
recall that by Lemma \ref{L:EmbedEdges} there is a bilipschitz
embedding of $MG$ into $X$, we denote this embedding by
$\varphi_0$. We may and shall assume that $\lip(\varphi_0)\le 1$
and $\lip(\varphi_0^{-1})$ is bounded from above by an absolute
constant. Next, we number vertices along short paths using numbers
from $1$ to $e(G)$ in such a way that vertices numbered $1$
correspond to the same long path in the correspondence described
above.
\medskip

At this point we are ready to describe the action of the map
$\varphi$ on vertices of short paths. For vertex $w$ of $H$ having
number $i$ on the short path $p_u$ the image in
$X\oplus_1\mathbb{R}$ is $\varphi_0(u)\oplus i$ (here we use the
same notation $u$ both for a vertex of $G$ and the corresponding
vertex in $MG$).
\medskip

To map vertices of long paths of $H$ into $X\oplus_1\mathbb{R}$ we
observe that the numbering of vertices of short paths leads to a
one-to-one correspondence between long paths and numbers
$\{1,\dots,e(G)\}$. We define the map $\varphi$ on a long path
corresponding to $i$ by $\varphi(w)=\varphi_0(w')\oplus i$, where
$w'$ is the uniquely determined vertex in a long path of $MG$
corresponding to a vertex $w$ in a long path of $H$.
\medskip

The fact that $\lip(\varphi)\le 1$ follows immediately from the
easily verified claim that the distance between $\varphi$-images
of adjacent vertices of $H$ is at most $1$ (here we use
$\lip(\varphi_0)\le 1$).\medskip

We turn to an estimate of $\lip(\varphi^{-1})$. In this part of
the proof we assume that $M>2e(G)$. Let $w$ and $z$ be two
vertices of $H$. As we have already mentioned our construction
implies that there are uniquely determined corresponding vertices
$w'$ and $z'$ in $MG$.\medskip

Obviously there are two possibilities:
\smallskip

(1) $d_{MG}(w',z')\ge\frac12d_H(w,z)$. In this case we observe
that the definitions of $\varphi$ and of the norm on
$X\oplus_1\mathbb{R}$ imply that
\[||\varphi(w)-\varphi(z)||\ge||\varphi_0(w')-\varphi_0(z')||\ge
d_{MG}(w',z')/\lip(\varphi_0^{-1})\ge\frac12d_H(w,z)/\lip(\varphi_0^{-1}).\]

(2) $d_{MG}(w',z')<\frac12d_H(w,z)$. This inequality implies that
there is a path joining $w$ and $z$ for which the naturally
defined {\it short-paths-portion} is longer than the {\it
long-paths-portion}. The inequality $M>2e(G)$ implies that the
short-paths-portion of this path consists of one path of length
$>\frac12d_H(w,z)$. This implies that the difference between the
second coordinates of $w$ and $z$ in the decomposition
$X\oplus_1\mathbb{R}$ is $>\frac12d_H(w,z)$. Thus
$||\varphi(w)-\varphi(z)|| >\frac12d_H(w,z)$.
\medskip

Since $\lip(\varphi_0^{-1})\ge 1$ (this follows from the
assumption $\lip(\varphi_0)\le 1$), we get $\lip(\varphi^{-1})\le
2\lip(\varphi_0^{-1})$ in each of the cases (1) and (2).
\end{proof}

\section{Proof of Lemma \ref{L:EmbedEdges}}

\begin{proof}
It is clear that it suffices to consider the case $\delta=1$. It
is convenient to handle all $M\in \mathbb{N}$ simultaneously by
considering the following thickening of the graph $G=G(X,1,r)$
(see \cite[Section 1.B]{Gro93} for the general notion of
thickening). For each edge $uv$ in $G$ we join $u$ and $v$ with a
set isometric to $[0,1]$, we denote this set $t(uv)$. The
thickening $TG$ is the union of all sets $t(uv)$ (such sets can
intersect at their ends only) with the distance defined as the
length of the shortest curve joining the points.\medskip

In order to prove Lemma \ref{L:EmbedEdges} it suffices to show
that there is a bilipschitz embedding of $TG$ into $X$ with
distortion bounded by an absolute constant. This follows
immediately from the observation that the graph $MG$ with the
scaled distance $\frac1Md_{MG}(u,v)$ is isometric to a subset of
$TG$.
\medskip

Recall that vertices of $G$ are elements of $X$. The restriction
of our embedding of $TG$ into $X$ to $V(G)$ will be the identical
map. So we need to define the embedding for edges only. We start
by looking at the following straightforward approach: map $t(uv)$
onto the line segment $[u,v]$ in such a way that the point in
$t(uv)$ which is at distance $\alpha\in(0,1)$ from $u$ is mapped
onto the point $\alpha u+(1-\alpha)v\in[u,v]$ (where $[u,v]$
denotes the line segment joining $u$ and $v$ in $X$). It is clear
that in general this straightforward approach does not have to
work: it can happen that $[u,v]$ intersects the line segment
$[u_1,v_1]$ corresponding to some other edge. In such a case the
straightforward embedding is not bilipschitz. Recall also that we
need to bound the distortion of the bilipschitz embedding of $TG$
into $X$ by an absolute constant.
\medskip

It turns out that the following perturbation of the
straightforward construction works. Recall that $\delta=1$. Let
$\mu=\frac14$ and let $z$ be the midpoint of $[u,v]$. Let
$B(z,\mu)$ denote the ball of radius $\mu$ centered at $z$. Our
purpose if to show that $B(z,\mu)$ contains a point $w$, such that
mapping the edge $t(uv)$ onto the union of line segments $[u,w]$
and $[w,v]$, we get a bilipschitz embedding whose distortion is
bounded by an absolute constant. Here we mean the map which maps
the point of $t(uv)$ lying at distance $\alpha$ from $u$ onto the
point in the curve obtained by concatenation of the line segments
$[u,w]$ and $[w,v]$ which is at along-the-curve distance
$\alpha(||w-u||+||v-w||)$ from $u$.\medskip

The map defined in this way is $4$-Lipschitz because it is clear
that $||w-u||+||v-w||\le ||z-u||+\frac14+||v-z||+\frac14\le
3.5<4$.
\medskip

To get a suitable estimate for the Lipschitz constant of the
inverse map we need to find $w$ in such a way that the line
segments $[u,w]$ and $[w,v]$ do not pass ``too close'' to line
segments corresponding to other edges.
\medskip

We make the notion of ``not-too-close'' more precise as follows.
We pick three numbers $\alpha$, $\beta$, and $\gamma$ in the
interval $\left(0,\frac14\right)$ satisfying
$\alpha,\gamma<\beta<\mu$. Some more restrictions will be
specified later, we also will show that numbers satisfying the
restrictions exist and can be chosen independently of $X$ and $r$.
\medskip

The ``not-too-close'' condition will be understood in the
following way. We find curves corresponding to different edges one
by one. When we turn to the construction of the curve
corresponding to the edge $t(uv)$, we do this in such a way that
the following conditions are satisfied:

\begin{itemize}

\item[$(\alpha)$]\label{I:alpha} The intersection of the curve
corresponding to $t(uv)$ with $B(u,\beta)$ is a line segment, and
the distance from the intersection of this line segment with the
sphere $S(u,\beta)$ (of radius $\beta$ centered at $u$) to
intersections with $S(u,\beta)$ with the line segments
corresponding to previously embedded edges is at least $\alpha$;
and the same condition holds for $v$.

\item[$(\beta)$] \label{I:beta} The curve corresponding to $t(uv)$
does not intersect $B(\widetilde u,\beta)$ for all vertices
$\widetilde u$ other than $u$ and $v$.

\item[$(\gamma)$] \label{I:gamma} The $\gamma$-neighborhoods of
those parts of curves corresponding to edges which are not in the
$\beta$-balls centered at vertices do not intersect curves
corresponding to other edges.

\end{itemize}

The first part of our proof consists of showing that if the
numbers $\alpha$, $\beta$, and $\gamma$ are sufficiently small
(but this ``smallness'' is independent on $X$ and its dimension
provided $\dim X\ge3$) and satisfy certain relations, then the set
of points $x\in B(z,\mu)$ for which at least one of the line
segments $[u,x]$ and $[x,v]$ does not meet the conditions above
does not exhaust $B(z,\mu)$. More precisely, we show that the
volume of the set of not-suitable points is smaller than the
volume of $B(z,\mu)$.
\medskip

The second part of the proof consists in showing that conditions
$(\alpha)$--$(\gamma)$ imply an absolute-constant upper estimate
of the distortion of the constructed map of $TG$ into $X$.
\medskip

Let $[u,\widetilde w]$ be one of the line segments in the image of
an already embedded edge $t(u\widetilde v)$ (recall that we map
each edge onto a union of two line segments). Our first goal is to
estimate from above the volume of the set of those points $x\in
B(z,\mu)$ for which the line segments $[u,x]$ and $[u,\widetilde
w]$ violate the condition $(\alpha)$. Observe that
\begin{equation}\label{E:LengthX-U}\frac12||v-u||-\mu\le ||x-u||\le\frac12||v-u||+\mu<2.\end{equation} If
condition $(\alpha)$ is violated, then
$||x-y||\le\alpha\cdot\frac{||x-u||}\beta$ for some $y$ in the ray
$\overrightarrow{u\widetilde w}$ (this is our notation for the ray
which starts at $u$ and passes through $\widetilde w$) satisfying
$||y-u||=||x-u||$. Therefore all vectors $x\in B(z,\mu)$ which are
``too close'' to the ray $\overrightarrow{u\widetilde w}$, are
contained in the $\frac{2\alpha}\beta$-neighborhood $T_1$ of a
line segment of length at most $2\mu$ (contained in the ray
$\overrightarrow{u\widetilde w}$).\medskip

\noindent{\bf Convention.} Everywhere in this proof by a {\it
volume} of a subset of $X$ we mean its Haar measure normalized in
such a way that the volume of the unit ball $B(0,1)$ is equal to
$1$.
\medskip

To estimate the volume of the neighborhood $T_1$ we observe that a
line segment of length $2\mu$ has a $\frac2\beta\alpha$-net of
cardinality $\left\lceil\frac{\mu\beta}{2\alpha}\right\rceil$. As
we shall see later, we have enough freedom in choosing $\alpha$,
to assume that $\frac{\mu\beta}{2\alpha}$ is an integer.
\medskip

The triangle inequality implies that the union of balls with radii
$\frac{4\alpha}{\beta}$ centered at all elements of the net covers
$T_1$. Hence
\[\vol(T_1)\le\frac{\mu\beta}{2\alpha}\cdot\left(\frac{4\alpha}\beta\right)^n=
2\mu\cdot\left(\frac{4\alpha}{\beta}\right)^{n-1},\] where $n$ is
the dimension of $X$.\medskip

Now we estimate from above the number of sets $T_1$ of the
described type which should be avoided when we try to find a
suitable location for $t(uv)$. It is clear that the number of such
sets is estimated from above by the degree of $u$ in $G$. The
estimate of the degree is standard (see e.g. \cite[Lemma
2.6]{MS86}): We need to estimate from above the number $N$ of
$1$-separated points in a ball of radius $3$ (see the definition
of $G(X,\delta,r)$). Interiors of balls of radii $\frac12$
centered at $1$-separated points do not intersect and are inside
the ball of radius $3+\frac12=\frac72$. Hence
$N\left(\frac12\right)^n\le\left(\frac72\right)^n$ and $N\le 7^n$.
\medskip

Thus the volume of the set which we have to exclude from
$B(z,\mu)$ in order to satisfy the condition $(\alpha)$ for both
$u$ and $v$ is
\[\le 2\cdot7^n\cdot(2\mu)\cdot\left(\frac{4\alpha}{\beta}\right)^{n-1}.\]

We would like this quantity to be less than
$\frac14\vol(B(z,\mu))$. This is achieved if
\[\frac14\mu^n\ge2\cdot7^n\cdot(2\mu)\cdot\left(\frac{4\alpha}{\beta}\right)^{n-1}.\]
Since $n\ge 3$ and $\mu=\frac14$, it is easy to verify that any
pair $\alpha$, $\beta$ satisfying
\begin{equation}\label{E:alphaEx}
\alpha\le\frac1{1232}\,\beta,
\end{equation}
satisfies the condition above.\medskip

Now we estimate from above the volume of the set of points $x\in
B(z,\mu)$ for which the curve obtained by concatenation of the
line segments $[u,x]$ and $[x,v]$ does not satisfy the condition
$(\beta)$. Recall that any vertex $y$ of $G=G(X,1,r)$ which is
different from $u$, is not contained in $B(u,1)$. Therefore, if
$||y-\widetilde x||\le\beta$ for some $\widetilde x\in[u,x]$, then
$||\widetilde x-u||\ge(1-\beta)$. Since $||x-u||\le 2$ (see
\eqref{E:LengthX-U}), we get that the distance between $x$ and
some point on the ray $\overrightarrow{uy}$ is
$\le\frac2{1-\beta}\cdot\beta$. Therefore the set of points in
$B(z,\mu)$ which violates the condition $(\beta)$ for given $y\in
V(G)$ is covered by $\frac2{1-\beta}\cdot\beta$-neighborhood $T_2$
of a line segment of length
$2\mu+2\left(\frac{2\beta}{1-\beta}\right)$. Our estimate of
$\vol(T_2)$ is similar to the estimate of $\vol(T_1)$. Let us
sketch it briefly. To simplify the estimate we assume that
$\beta<\mu/5$. In such a case
$2\mu+2\left(\frac{2\beta}{1-\beta}\right)< 3\mu$ and
$\left(\frac{2\beta}{1-\beta}\right)<3\beta$. So it suffices to
estimate the volume of $3\beta$-neighborhood of the line segment
of length $3\mu$. In the same way as before we get the estimate
$\le\frac{3\mu}{6\beta}\cdot(6\beta)^n$.\medskip

We need to estimate the number of vertices $y$ for which such sets
have to be excluded. Here, for simplicity we can use the same
estimate $\le 7^n$ because it is clear that a vertex whose
distance from $u$  in $X$ is $>3$ cannot ``stay on the way'' of
the line segment $[u,x]$, $x\in B(z,\mu)$.\medskip

We get the following upper estimate of the volume of the part of
$B(z,\mu)$ which should be excluded in order to eliminate all
points $x$ for which the curve obtained by concatenation of
$[u,x]$ and $[x,v]$ violates the condition $(\beta)$. The volume
does not exceed $2\cdot 7^n\cdot(3\mu)\cdot(6\beta)^{n-1}$. Again
we would like this quantity to be less than
$\frac14\,\vol(B(z,\mu))=\frac14\,\mu^n$. Recalling that $n\ge 3$
we get that the condition is satisfied for
\begin{equation}\label{E:betaEx}\beta\le\mu/546.\end{equation} This is our second requirement on
the triple $(\alpha,\beta,\gamma)$.\medskip

Now we turn to estimates the volume of the sets of those $x\in
B(z,\mu)$ for which the concatenation of $[u,x]$ and $[x,v]$ does
not satisfy condition $(\gamma)$. This happens only if either
$[u,x]$ or $[x,v]$ intersects a $\gamma$-neighborhood of an
already embedded into $X$ edge $f(t(\widetilde u\widetilde v))$
($f$ denotes the embedding). We do the estimates for the case when
$[u,x]$ intersects a $\gamma$-neighborhood of an already embedded
into $X$ edge $f(t(\widetilde u\widetilde v))$; at the end we
multiply the obtained estimate for the volume of non-suitable
points by $2$. First we estimate the number of such edges
$\widetilde u\widetilde v$. In the described situation both
$\widetilde u$ and $\widetilde v$ should be in a ball of radius
$5$ centered at $u$ (recall that $\gamma\le\frac14$). In the same
way as we estimated the number of vertices in a $3$-ball, we get
an estimate $\le 11^n$ for the number of vertices $\widetilde u$
for which $f(t(\widetilde u\widetilde v))$ can occur as an
obstacle. Hence the number of edges which could interfere with the
line segment $[u,x]$ is $\le 121^n$.
\medskip

Now we estimate from above the volume of those $x\in B(z,\mu)$
which have to be excluded because of one already embedded edge
$f(t(\widetilde u\widetilde v))$. First recall that the length of
each such edge is $\le 4$ and that we do not need to care about
its portion which is inside $B(u,\beta)$.
\medskip

For each point $y\in f(t(\widetilde u\widetilde v))$ satisfying
$||y-u||\ge\beta$ we consider the ray $\overrightarrow{uy}$. If
$y$ and $[u,x]$ violate the condition $(\gamma)$, then there is
$\widetilde x\in[u,x]$ such that $||y-\widetilde x||\le\gamma$.
Then $||\widetilde x-u||\ge\beta-\gamma$. Since $||x-u||\le 2$, we
get that the distance between $x$ and some point on the ray
$\overrightarrow{uy}$ is $\le\frac2{\beta-\gamma}\cdot\gamma$.
Therefore the set of points in $B(z,\mu)$ for which $[u,x]$ comes
``too close'' to the point $y\in f(t(\widetilde u\widetilde v))$
is covered by $\frac{2\gamma}{\beta-\gamma}$-neighborhood $T_3(y)$
of a line segment of length $\le2\mu+\frac{4\gamma}{\beta-\gamma}$
on the ray $\overrightarrow{uy}$.
\medskip

For simplicity of the remaining argument we assume that
$\gamma\le\frac1{20}\,\beta$. In such a case
$\frac{4\gamma}{\beta-\gamma}\le\mu$ and
$\frac{2\gamma}{\beta-\gamma}\le\frac{3\gamma}{\beta}$. Thus
$T_3(y)$ is covered by a $\frac{3\gamma}\beta$-neighborhood of a
line segment $\ell(y)$ of length $3\mu$. Furthermore, it is easy
to see that we may assume that the center of $\ell(y)$ is at
distance $\frac12||v-u||$ from $u$. We need to estimate from above
the volume of
\[T_3:=\bigcup_{y\in f(t(\widetilde u\widetilde v))}T_3(y).\]

We need the following observation: Let $h_1$ and $h_2$ be two
vectors in $X$ satisfying $||h_1||\ge||h_2||>0$. By the triangle
inequality we have
\begin{equation}\label{E:TriNorm}\left\|h_1-\frac{||h_1||}{||h_2||}\,
h_2\right\|\le2||h_1-h_2||.\end{equation}

Let $y_1$ and $y_2$ be two points on the curve $f(t(\widetilde
u\widetilde v))$ with $||y_1-u||\ge||y_2-u||\ge\beta$.  Using
\eqref{E:TriNorm} for $h_1=y_1-u$ and $h_2=y_2-u$ and homothety we
get that the distance between a point of $\ell(y_1)$ and the point
of $\ell(y_2)$ with the same distance to $u$ is
\begin{equation}\label{E:DistCent}
\le\frac2\beta\cdot2||y_1-y_2||=\frac4\beta\,||y_1-y_2||.\end{equation}

Since the length of $f(t(\widetilde u\widetilde v))$ is $\le 4$,
there is a $\gamma$-net of cardinality $\le\frac4{2\gamma}$ in
$f(t(\widetilde u\widetilde v))$. Also there is a
$\frac\gamma\beta$-net of cardinality
$\le\frac{3\mu\beta}{2\gamma}$ in any line segment of length
$3\mu$. Combining these nets and using inequality
\eqref{E:DistCent} we get a set $\mathcal{M}$ satisfying the
following conditions: (a)
$|\mathcal{M}|\le\frac{3\mu\beta}{\gamma^2}$; (b) For each
$q\in\bigcup_{y\in f(t(\widetilde u\widetilde v))}\ell(y)$ there
is $q_0\in\mathcal{M}$ with $||q-q_0||\le\frac{5\gamma}{\beta}$.
\medskip

Therefore balls of radii $\frac{8\gamma}{\beta}$ centered at
elements of $\mathcal{M}$ cover $T_3$. We get the estimate
\[\vol(T_3)\le\frac{3\mu\beta}{\gamma^2}\cdot\left(\frac{8\gamma}\beta\right)^n.\]

As we have already mentioned, we should consider $\le 121^n$
already embedded edges, also we need to consider the line segments
$[x,v]$ as well. Thus, the volume of all points $x$ in $B(z,\mu)$
which fail to satisfy $(\gamma)$ is estimated from above by

\[2\cdot121^n\cdot\frac{3\beta\mu}{\gamma^2}\cdot\left(\frac{8\gamma}\beta\right)^n.\]

As in the previous estimates, we would like this quantity to be
$\le\frac14\vol(B(z,\mu))$, that is, we need the inequality
\[2\cdot121^n\cdot\frac{3\beta\mu}{\gamma^2}\cdot\left(\frac{8\gamma}\beta\right)^n\le\frac14\,\mu^n\]
to hold. This inequality can be rewritten as
\[\gamma^{n-2}\le C\mu^{n-1}\beta^{n-1}\left(\frac1{8\cdot
121}\right)^{n-2},\] where $C$ is an absolute constant.\medskip

Since $n\ge 3$ we have $n-2\ge 1$ and $n-1\le2(n-2)$. Because of
this and because $\beta,\mu\in(0,1)$, we have $(\beta\mu)^{n-1}\ge
(\beta\mu)^{2(n-2)}$. Therefore it suffices to satisfy
\[\gamma^{n-2}\le C\left(\frac{(\beta\mu)^2}{8\cdot121}\right)^{n-2}.\]
Using again the inequality $n\ge 3$, we see that it suffices to
pick
\begin{equation}\label{E:gammaEx}
\gamma\le C\frac{(\beta\mu)^2}{8\cdot121}.\end{equation}

Now it is clear that we can choose $\alpha$, $\beta$, and $\gamma$
satisfying the conditions \eqref{E:alphaEx}, \eqref{E:betaEx}, and
\eqref{E:gammaEx}. We start with choosing $\beta$ satisfying
\eqref{E:betaEx}, then we choose $\alpha$ satisfying
\eqref{E:alphaEx} and $\gamma$ satisfying \eqref{E:gammaEx}. It is
clear that we may assume $\gamma\le\alpha$ and that we had right
to make the assumptions on relations between $\alpha$, $\beta$,
and $\gamma$ which we made in our proof.\medskip
\medskip

Thus, at each step it is possible to find $w\in B(z,\mu)$ such
that the curve obtained by concatenation of $[u,w]$ and $[w,v]$
satisfies the assumptions $(\alpha)$--$(\gamma)$.\medskip

To complete the proof of Lemma \ref{L:EmbedEdges} it suffices to
estimate the Lipschitz constant of the inverse map by an absolute
constant. So we need to consider two points $u,v$ in the image of
$TG$ and to estimate from above the ratio
\begin{equation}\label{E:ratio}
\frac{d_{TG}(u,v)}{||f(u)-f(v)||}.
\end{equation}

The estimate in the case when both points are vertices is given in
Lemma \ref{L:GraphInBall}. Next we consider the case when one of
the points, say $u$, is a vertex. There are two subcases: (1) The
second point is on the edge incident to $u$; (2) The second point
is on the edge which is not incident to $u$.\medskip

Subcase (1): If $||f(u)-f(v)||\le\beta$, then the corresponding
portion of the edge is a line segment. Since edges of length $1$
are represented by curves whose length is $>1$, this implies
$d_{TG}(u,v)\le||f(u)-f(v)||$. If $||f(u)-f(v)||\ge\beta$, then,
since the distance in $TG$ is $\le 1$, the ratio \eqref{E:ratio}
is $\le\frac1\beta$.
\medskip

Subcase (2): The second point $v$ is not on an edge incident with
$u$. Let $D=||f(u)-f(v)||$. Then the distance in $X$ between $u$
and one of the ends of the edge to which $v$ belongs is $\le D+2$
(recall that vertices of $G$ are identified with elements of $X$,
so $u=f(u)$). By the proof of Lemma \ref{L:GraphInBall}, there is
a path from that end to $u$ of length $\le D+2$. Hence
$d_{TG}(u,v)\le D+3$. Since by condition $(\beta)$ we have $D\ge
\beta$, the ratio \eqref{E:ratio} in this case does not exceed
\[\frac{D+3}D\le \frac{\beta+3}\beta=1+\frac3\beta.\]
\medskip

Now we consider the situation when $u$ and $v$ are on different
edges. Subcases: (1) The edges are adjacent; (2) The edges are not
adjacent.\medskip

Subcase (1): Subsubcase (a) One of the points is outside the
$\beta$-ball centered at the common end. Then, by  condition
$(\gamma)$, we have $||f(u)-f(v)||\ge\gamma$. Since in $TG$ the
distance between two points belonging to adjacent edges is $\le
2$, we get that the ratio \eqref{E:ratio} is $\le\frac2\gamma$.
\medskip

Subsubcase (b) Both points are inside the $\beta$-ball centered at
some vertex $q=f(q)$.  Let
$||f(u)-f(q)||\le||f(v)-f(q)||=\omega\beta$ for some
$\omega\in(0,1]$. Using condition $(\alpha)$ and a simple
geometric argument we get that the distance between $f(v)$ and the
ray $\overrightarrow{f(q)f(u)}$ is $\ge\frac{\alpha\omega}2$. On
the other hand,  $d_{TG}(u,v)\le 2\omega\beta$. Therefore the
ratio \eqref{E:ratio} in this case is
$\le\frac{2\omega\beta}{\frac{\alpha\omega}2}=\frac{4\beta}\alpha$.
\medskip

Subcase (2) The points $u,v$ are on non-adjacent edges. Then
$||f(u)-f(v)||\ge\gamma$. Let $D=||f(u)-f(v)||$. Then the distance
between two of the ends of the edges in $X$ is $\le D+4$. Hence in
$TG$ it is also $\le D+4$. Hence the total distance between the
points in the graph is $\le D+6$, and the ratio \eqref{E:ratio} is
\[\le\frac{D+6}D\le\frac{\gamma+6}{\gamma}\le1+\frac6\gamma.\]
Taking into account the fact that we may assume that
$\gamma\le\alpha\le\beta<1$, we get that the Lipschitz constant of
the inverse map is $\le1+\frac6\gamma$.
\end{proof}

\section{Unweighted graphs admitting isometric embeddings into
strictly convex Banach spaces}\label{S:StricConv}

\begin{observation}\label{O:GraphSC} If a finite simple connected graph $G$ endowed with its graph
distance admits an isometric embedding into a strictly convex
Banach space $X$, then $G$ is isomorphic to either a complete
graph or a path.
\end{observation}

\begin{proof} Assume that $G$
is a finite simple connected graph, which is not a path, but is
such that $(V(G),d_G)$ is isometric to a subset of a strictly
convex space $X$ (we use the standard definition of strict
convexity, see \cite[p.~409]{BL00}). Denote the isometric
embedding by $f$. Our goal is to show that these conditions imply
that $G$ is a complete graph.
\medskip

The fact that $G$ is not a path immediately implies that $G$ is
either a cycle or has a vertex of degree $3$. In the case when $G$
is a cycle we observe that the cycle $C_3$ is simultaneously a
complete graph $K_3$, and we are done in this case. As for longer
cycles we prove that they do not admit isometric embeddings into
$X$ in the following way. Since vertices $v_{k-1}, v_{k}, v_{k+1}$
in a cycle satisfy
$d_G(v_{k-1},v_{k+1})=d_G(v_{k-1},v_{k})+d_G(v_{k},v_{k+1})$, by
strict convexity we get that $f(v_{k-1}), f(v_{k})$, and
$f(v_{k+1})$ should be on the same line, with $f(v_k)$ being a
midpoint of the line segment $[f(v_{k-1}),f(v_{k+1})]$. Since this
observation is applicable also to $v_{n-1},v_n,v_1$ and
$v_n,v_1,v_2$, we get a contradiction.
\medskip

Now let $v\in V(G)$ be a vertex of degree $\ge 3$, and let
$u_1,u_2,u_3$ be its neighbors. We show that $u_i$ are pairwise
adjacent. If two pairs of them (say $u_1,u_2$ and $u_2,u_3$) are
not adjacent, we get a contradiction because $f(v)$ should be
simultaneously a midpoint of the line segment joining $f(u_1)$ and
$f(u_2)$ and a midpoint of the line segment joining $f(u_2)$ and
$f(u_3)$.
\medskip

If only one edge, say $u_1u_3$, is missing then both $f(u_2)$ and
$f(v)$ should be midpoints of the line segment joining  $f(u_1)$
and $f(u_3)$.
\medskip

Therefore $v$ and all of its neighbors should form a complete
subgraph in $G$. Since the same should hold for each of the
neighbors of $v$, we get that $G$ should be a complete graph.
\end{proof}

The author wishes to thank the referee for the helpful and
constructive criticism of the first version of the paper.

\end{large}

\begin{small}

\renewcommand{\refname}{
\end{small}
\medskip

\noindent{\sc Department of Mathematics and Computer Science\\
St. John's University\\ 8000 Utopia Parkway, Queens, NY 11439,
USA}\\
e-mail: {\tt ostrovsm@stjohns.edu}

\end{document}